\documentclass[reqno]{amsart}
\usepackage{amssymb,latexsym,amsmath,amsthm,enumerate,amsbsy}
\usepackage[mathscr]{eucal}
\usepackage{framed,color,graphicx}
\usepackage{mathrsfs}
\usepackage[all]{xy}
\usepackage{tikz}
\usepackage{cite}
\usepackage{longtable}
\usepackage{ltxtable}
\usepackage{subcaption}
\usepackage{graphicx}
\usepackage{subfloat}
\usepackage{subfig}
\usepackage{tikz}
\usepackage{caption} 
\usepackage{etoolbox} 
\usepackage{tabularx}
\usepackage{eqparbox}
\usepackage{array}

\usetikzlibrary{positioning,decorations.pathreplacing,patterns,decorations.pathmorphing}
\tikzset{%
element/.style={draw, shape=circle, fill=white, inner sep=1.4pt}
}

\DeclareSymbolFont{bbold}{U}{bbold}{m}{n}
\DeclareSymbolFontAlphabet{\mathbbold}{bbold}

\theoremstyle{plain}
\newtheorem{thm}{Theorem}[section]
\newtheorem{lem}[thm]{Lemma}

\newtheorem{cor}[thm]{Corollary}
\newtheorem{pro}[thm]{Proposition}

\theoremstyle{definition}

\newtheorem{remark}[thm]{Remark}

\newcounter{eqbasis}
\renewcommand{\theeqbasis}{E\arabic{eqbasis}}
\newcommand{\eqtag}{\refstepcounter{eqbasis}(\theeqbasis)}

\newcommand{\eqrefbasis}[1]{\textup{(\ref{#1})}}

\renewcommand{\ge}{\geqslant}

\newcommand{\bp}{\mathbf{p}}
\newcommand{\bq}{\mathbf{q}}

\newcommand{\bu}{\mathbf{u}}
\newcommand{\bv}{\mathbf{v}}
\newcommand{\bw}{\mathbf{w}}

\begin{document}
\title[The finite basis problem for matrix semirings]
{The finite basis problem for matrix semirings over a two-element additively idempotent semiring}

\author{Jun Jiao}
\address{School of Mathematics, Northwest University, Xi'an, 710127, Shaanxi, P.R. China}
\email{jjunjiao@163.com}

\author{Miaomiao Ren}
\address{School of Mathematics, Northwest University, Xi'an, 710127, Shaanxi, P.R. China}
\email{miaomiaoren@yeah.net}

\dedicatory{Dedicated to Professor Mikhail V. Volkov on the occasion of his 70th birthday}

\subjclass[2010]{16Y60, 03C05, 08B05}
\keywords{semiring, matrix, variety, identity, finite basis problem}
\thanks{Miaomiao Ren, corresponding author, is supported by National Natural Science Foundation of China (12371024, 12571020).}

\begin{abstract}
We provide a complete classification of matrix semirings $\mathbf{M}_n(S)$ over two-element additively idempotent semirings $S$
with respect to the finite basis property.
Our main theorem shows that for every integer $n \geq 2$,
the semiring $\mathbf{M}_n(S)$ is finitely based
if and only if $S$ is distinct from a distributive lattice.
\end{abstract}

\maketitle

\section{Introduction}
An \emph{additively idempotent semiring} (abbreviated as ai-semiring)
is an algebraic structure $(S, +, \cdot)$ equipped with two binary operations $+$ and $\cdot$ such that:
\begin{itemize}
\item the additive reduct $(S, +)$ is a commutative idempotent semigroup;
\item the multiplicative reduct $(S, \cdot)$ is a semigroup;
\item the distributive laws hold:
\[
(x+y)z \approx xz + yz, \qquad x(y+z) \approx xy + xz.
\]
\end{itemize}
This class of algebras includes well-known examples such as the Kleene semiring of regular languages~\cite{con},
the max-plus algebra~\cite{aei}, the power semiring of a semigroup~\cite{dgv24},
the endomorphism semiring of a semilattice~\cite{dgv25},
the semiring of all binary relations on a set \cite{dolinka2009}, and distributive lattices~\cite{burris1981}.
These and other similar algebras have found significant applications in several branches of mathematics,
such as algebraic geometry~\cite{cc}, tropical geometry~\cite{ms}, information science~\cite{gl},
theoretical computer science~\cite{go}, and soft constraint solving and programming~\cite{bis2004}.

Let $S$ be an ai-semiring. Define a binary relation $\leq$ on $S$ by
\[
a\leq b\Leftrightarrow a+b=b.
\]
Then $\leq$ is a partial order on $S$, and $(S, \leq)$ becomes an upper semilattice in which
the supremum of any elements $a$ and $b$ is $a+b$.
Consequently, the additive reduct $(S, +)$ is uniquely determined by this semilattice order.
Therefore, it is often convenient to represent the addition operation using the Hasse diagram of $(S, \leq)$.
Moreover, one can easily verify that this semilattice order is compatible with multiplication.
For this reason, an ai-semiring is often called a \emph{semilattice-ordered semigroup}.

Let $n\geq 1$ be an integer,
let $S$ be an ai-semiring, and let $\mathbf{M}_n(S)$ denote the set of all $n \times n$ matrices over $S$.
Then $\mathbf{M}_n(S)$ forms an ai-semiring under the usual matrix addition and multiplication, that is,
\[
A+B = [a_{ij}+b_{ij}]_{n\times n}, \quad A \cdot B = \left[ \sum_{k=1}^n a_{ik} b_{kj} \right]_{n\times n}
\]
for all $A=[a_{ij}]_{n\times n}, B=[b_{ij}]_{n\times n}\in \mathbf{M}_n(S)$.
It is obvious that $S$ is isomorphic to $\mathbf{M}_1(S)$.
For $n\geq 2$, $S$ can be embedded into $\mathbf{M}_n(S)$,
since the mapping
\[
\varphi\colon S \rightarrow \mathbf{M}_n(S),\quad a \mapsto [a]_{n\times n},
\]
where $[a]_{n\times n}$ denotes the constant matrix with all entries equal to $a$,
is an embedding mapping.
Moreover, if $S$ has an additive identity that is also a multiplicative zero,
then we can easily obtain the following infinite strict ascending chain
\[
S\hookrightarrow \mathbf{M}_2(S)\hookrightarrow \mathbf{M}_{3}(S)\hookrightarrow  \mathbf{M}_{4}(S)\hookrightarrow\cdots.
\]


Let $\mathcal{V}$ be a class of algebras.
Then $\mathcal{V}$ is called a \emph{variety} if it is closed under
taking subalgebras, homomorphic images, and arbitrary direct products.
By the celebrated Birkhoff's theorem \cite{birkhoff1935}, $\mathcal{V}$ is a variety if and only if
it is an \emph{equational class}, that is, the class of all algebras satisfying some set of identities.

A variety is \emph{finitely based} if it can be defined by a finite set of identities;
otherwise, it is \emph{nonfinitely based}.
An algebra $A$ is finitely based or nonfinitely based if the variety $\mathsf{V}(A)$
it generates is finitely based or not.
The \emph{finite basis problem} for a class of algebras,
one of the central problems in universal algebra,
concerns the classification of
its members with respect to the property of being finitely based.

A variety is \emph{locally finite} if each of its finitely generated members is finite.
A locally finite variety is \emph{inherently nonfinitely based}
if it is not contained in any finitely based locally finite variety.
Since every finite algebra generates a locally finite variety,
we say that a finite algebra $A$ is inherently nonfinitely based if
the variety $\mathsf{V}(A)$ is inherently nonfinitely based.
It follows immediately that every finite algebra whose variety contains
an inherently nonfinitely based algebra is nonfinitely based.
Consequently, every inherently nonfinitely based algebra must be nonfinitely based.

Over the past two decades,
the finite basis problem for ai-semirings has attracted considerable attention,
resulting in substantial
progress~\cite{dol07, dolinka2009, dgv24, dgv25, gpz05, gv23, gv2501, gv2510, jrz, pas05, rz16, rzw, sr, volkov2024, yr25, yrzs, zrc}.
In particular, Dolinka~\cite{dol07} constructed the first example of a nonfinitely based finite ai-semiring.
Pastijn et al.~\cite{gpz05, pas05} proved that every ai-semiring satisfying the identity $x^2 \approx x$ is finitely based;
Ren et al.~\cite{rz16, rzw} later extended this to ai-semirings satisfying the identity $x^3 \approx x$.
Shao and Ren~\cite{sr} established that every algebra in the variety generated by all two-element ai-semirings is finitely based.
Subsequently,
Jackson et al.~\cite{jrz} and Zhao et al.~\cite{zrc}
provided a complete classification of all three-element ai-semirings with respect to the finite basis property.
Most recently,
Dolinka, Gusev and Volkov~\cite{dgv25, gv2501} settled the finite basis problem for
the endomorphism semirings of finite semilattices.

To the best of our knowledge,
the study of the finite basis problem for matrix semirings was initiated by
Dolinka~\cite{dolinka2009}, who considered the problem for $\mathbf{M}_n(D_2)$ for $n\ge 2$,
where $D_2$ denotes the two-element distributive lattice.
Gusev and Volkov~\cite{gv2510, volkov2024} are currently investigating
the same problem for upper triangular matrix semirings over $D_2$.
Motivated by these works,
we undertake a systematic classification of matrix semirings  $\mathbf{M}_n(S)$,
where $S$ ranges over all two-element ai-semirings.
The main result of this paper provides the following complete characterization.

\begin{thm}\label{main}
Let $S$ be a two-element ai-semiring, and let $n \geq 2$ be an integer.
Then the matrix semiring $\mathbf{M}_n(S)$ is finitely based if and only if $S$ is not a distributive lattice.
\end{thm}

For the proof,
we first note that the case where $S$ is a distributive lattice follows from Dolinka's work.
Specifically,
\cite[Theorem B]{dolinka2009} states that the semiring $\mathcal{R}_2$ of all binary relations on a two-element set is inherently nonfinitely based.
Since $\mathcal{R}_2 $ is isomorphic to $\mathbf{M}_2(D_2)$,
it follows immediately that $\mathbf{M}_2(D_2)$ inherits this property.
Furthermore, because $\mathbf{M}_n(D_2)$ contains a copy of $\mathbf{M}_2(D_2)$ for all $n \geq 2$,
the same property extends to $\mathbf{M}_n(D_2)$; this result is explicitly stated in \cite[Corollary 6.2]{dolinka2009}.
(We note that $D_2$ is also denoted $\mathbb{B}_2$ and called the two-element Boolean semiring; our notation follows \cite{sr}.)

Therefore, to establish Theorem~\ref{main}, it remains to prove the converse:
that the matrix semiring $\mathbf{M}_n(S)$ is finitely based
for every two-element ai-semiring $S$ distinct from a distributive lattice and every $n \geq 2$.
The necessary preliminaries are collected in Section 2, and the proof is carried out in Sections 3 and 4.

\section{Preliminaries}

Let $X$ be a countably infinite set of variables, let $X^+$ denote the free semigroup over $X$,
and let $X^*$ denote the free monoid over $X$.
Due to distributivity, every ai-semiring term over $X$ can be expressed as a finite sum of words from $X^+$.
An ai-semiring identity over $X$ is an
expression of the form
\[
\bu\approx \bv,
\]
where $\bu$ and $\bv$ are ai-semiring terms over $X$.
By \cite[Theorem 2.5]{kp},
the ai-semiring $(P_f(X^+), \cup, \cdot)$ of all non-empty finite subsets of $X^+$ is free
in the variety $\mathbf{AI}$ of all ai-semirings on $X$.
So we sometimes write
\[
\{\bu_i \mid 1 \leq i \leq k\}\approx \{\bv_j \mid 1 \leq j \leq \ell\}
\]
for the ai-semiring identity
\[
\bu_1+\cdots+\bu_k\approx \bv_1+\cdots+\bv_\ell.
\]
For an ai-semiring $S$ and an ai-semiring identity $\bu\approx \bv$,
we say that $S$ \emph{satisfies} $\bu\approx \bv$ (or $\bu\approx \bv$ \emph{holds} in $S$)
if $\varphi(\bu)=\varphi(\bv)$
for every semiring homomorphism $\varphi\colon P_f(X^+) \to S$.
Note that such homomorphism $\varphi$ is uniquely determined
by its values on $X$, since $P_f(X^+)$ is generated by $X$.

Suppose that $\Sigma$ is a set of ai-semiring identities containing the defining identities of the variety $\mathbf{AI}$.
Let $\bu \approx \bv$ be an ai-semiring identity with
\[
\bu=\bu_1+\cdots+\bu_k \quad\textrm{and}\quad \bv=\bv_1+\cdots+\bv_\ell,
\]
where $\bu_i, \bv_j \in X^+$ for $1 \leq i \leq k$ and $1 \leq j \leq \ell$.
One can easily check that the ai-semiring variety defined by $\bu \approx \bv$
coincides with the variety defined by the simpler identities $\bu \approx \bu+\bv_j$ and $\bv \approx \bv+\bu_i$
for all $1 \leq i \leq k$ and $1 \leq j \leq \ell$.
Therefore, to prove that $\bu \approx \bv$ is derivable from $\Sigma$,
it is enough to derive each identity $\bu \approx \bu+\bv_j$ and $\bv \approx \bv+\bu_i$ from $\Sigma$.
This reduction will be used in the subsequent sections.

Next, we introduce some notation.
Let $\bw$ be a word in $X^+$ and $x$ a letter in $X$. Then
\begin{itemize}
\item $h(\bw)$ denotes the first variable that occurs in $\bw$;

\item $t(\bw)$ denotes the last variable that occurs in $\bw$;

\item $c(\bw)$ denotes the set of variables that occur in $\bw$;

\item $\ell(\bw)$ denotes
the number of variables occurring in $\bw$ counting multiplicities;

\item $L_{\geq 2}(\bu)$ denotes the set of $\bu_i \in \bu$ such that $\ell(\bu_i)\geq 2$;

\item $p(\bw)$ denotes the word obtained from $\bw$ by deleting its tail,
that is, $\bw=p(\bw)t(\bw)$;

\item $s(\bw)$ denotes the word obtained from $\bw$ by deleting its head,
that is, $\bw=h(\bw)s(\bw)$.
\end{itemize}

Let $\bu$ be an ai-semiring term such that $\bu=\bu_1+\bu_2+\cdots+\bu_n$, where
$\bu_i \in X^+$, $1\leq i \leq n$. Then
\begin{itemize}
\item $h(\bu)$ denotes the set $\{h(\bu_i) \mid 1\leq i \leq n\}$;

\item $t(\bu)$ denotes the set $\{t(\bu_i) \mid 1\leq i \leq n\}$;

\item $c(\bu)$ denotes the set of variables that occur in $\bu$ and so
\[
c(\bu)=\bigcup_{1\leq i \leq n} c(\bu_i).
\]
\end{itemize}

Up to isomorphism, there are exactly six ai-semirings of order two (see \cite{sr}),
which are denoted by $L_2$, $R_2$, $N_2$, $T_2$, $M_2$ and $D_2$.
For each of these semirings the carrier set is $\{0,1\}$, their addition and multiplication tables,
along with the equational bases, are presented in Table~\ref{tab:2-element-ai-semirings}.

%

\begin{table}[htbp]
\caption{The 2-element ai-semirings}\label{tab:2-element-ai-semirings}
\setlength{\tabcolsep}{10pt}
\begin{tabular}{cccc}
\hline
Semiring & $+$ & $\cdot$ & Equational basis \\
\hline
$L_2$ &
\begin{tabular}{cc}
0 & 1 \\
1 & 1 \\
\end{tabular} &
\begin{tabular}{cc}
0 & 0 \\
1 & 1 \\
\end{tabular} &
\begin{tabular}{r}
$xy \approx x$\quad\eqtag\label{eq:L2}
\end{tabular}
\\
\hline
$R_2$ &
\begin{tabular}{cc}
0 & 1 \\
1 & 1 \\
\end{tabular} &
\begin{tabular}{cc}
0 & 1 \\
0 & 1 \\
\end{tabular} &
\begin{tabular}{r}
$xy \approx y$\quad\eqtag\label{eq:R2}
\end{tabular}
\\
\hline
$N_2$ &
\begin{tabular}{cc}
0 & 1 \\
1 & 1 \\
\end{tabular} &
\begin{tabular}{cc}
0 & 0 \\
0 & 0 \\
\end{tabular} &
\begin{tabular}{r}
$x_1x_2 \approx y_1y_2$, \quad \eqtag \label{eq:N21}\\
$x + x^2 \approx x$      \quad \eqtag \label{eq:N22}
\end{tabular} \\
\hline
$T_2$ &
\begin{tabular}{cc}
0 & 1 \\
1 & 1 \\
\end{tabular} &
\begin{tabular}{cc}
1 & 1 \\
1 & 1 \\
\end{tabular} &
\begin{tabular}{r}
$x_1x_2 \approx y_1y_2$, \quad  \eqref{eq:N21}\\
$x + x^2 \approx x^2$~\quad \eqtag \label{eq:T22}
\end{tabular} \\
\hline
$M_2$ &
\begin{tabular}{cc}
0 & 1 \\
1 & 1 \\
\end{tabular} &
\begin{tabular}{cc}
0 & 1 \\
1 & 1 \\
\end{tabular} &
\begin{tabular}{r}
$x+y \approx xy$\quad\eqtag\label{eq:M2}
\end{tabular}
\\
\hline
$D_2$ &
\begin{tabular}{cc}
0 & 1 \\
1 & 1 \\
\end{tabular} &
\begin{tabular}{cc}
0 & 0 \\
0 & 1 \\
\end{tabular} &
\begin{tabular}{c}
$x^2 \approx x$, $xy \approx yx$, \\
$x + xy \approx x$
\end{tabular} \\
\hline
\end{tabular}
\end{table}

The following result, stated as \cite[Lemma 1.1]{sr},
completely solves the equational problem for all two-element ai-semirings.
We shall use this fact without further reference.

\begin{lem}\label{nlemma1}
Let $\bu\approx \bu+\bq$ be a nontrivial ai-semiring identity such that
$\bu=\bu_1+\cdots+\bu_n$, where $\bu_i, \bq\in X^+$, $1\leq i \leq n$. Then
\begin{itemize}
\item[$(1)$] $\bu\approx \bu+\bq$ holds in $L_2$ if and only if $h(\bq)=h(\bu_i)$ for some $\bu_i \in \bu$.

\item[$(2)$] $\bu\approx \bu+\bq$ holds in $R_2$ if and only if $t(\bq)=t(\bu_i)$ for some $\bu_i \in \bu$.

\item[$(3)$] $\bu\approx \bu+\bq$ holds in $N_2$ if and only if $\ell(\bq)\geq 2$.

\item[$(4)$] $\bu\approx \bu+\bq$ holds in $T_2$ if and only if $\ell(\bu_i)\geq 2$ for some $\bu_i \in \bu$.

\item[$(5)$] $\bu\approx \bu+\bq$ holds in $M_2$ if and only if $c(\bq) \subseteq \bigcup_{i=1}^n c(\bu_i)$.

\item[$(6)$] $\bu\approx \bu+\bq$ holds in $D_2$ if and only if $c(\bq)\supseteq c(\bu_i)$ for some $\bu_i \in \bu$.
\end{itemize}
\end{lem}

\begin{table}[htbp]
\caption{Some 3-element ai-semirings}\label{tb:3-element-ai-semirings}
\begin{tabular}{cccc}
\hline
Semiring & $+$ & $\cdot$ & Equational basis \\
\hline
$S_{54}$ &
\begin{tabular}{ccc}
1 & 1 & 3 \\
1 & 2 & 3 \\
3 & 3 & 3
\end{tabular} &
\begin{tabular}{ccc}
3 & 1 & 3 \\
3 & 2 & 3 \\
3 & 3 & 3
\end{tabular} &
\begin{tabular}{c}
$xyz\approx xzy$, $xy^2\approx xy$, \\
$x+zy\approx xy+zy$, $x^2+yx \approx yx$
\end{tabular} \\
\hline
$S_{56}$ &
\begin{tabular}{ccc}
1 & 1 & 3 \\
1 & 2 & 3 \\
3 & 3 & 3
\end{tabular} &
\begin{tabular}{ccc}
3 & 2 & 3 \\
3 & 2 & 3 \\
3 & 2 & 3
\end{tabular} &
\begin{tabular}{c}
$xy\approx zy$, \\
$x+x^2\approx x^2$
\end{tabular} \\
\hline
$S_{57}$ &
\begin{tabular}{ccc}
1 & 1 & 3 \\
1 & 2 & 3 \\
3 & 3 & 3
\end{tabular} &
\begin{tabular}{ccc}
3 & 3 & 3 \\
1 & 2 & 3 \\
3 & 3 & 3
\end{tabular} &
\begin{tabular}{c}
$xyz\approx yxz$, $x^2y\approx xy$, \\
$x+yz\approx yx+yz$, $x^2+xy \approx xy$
\end{tabular} \\
\hline
$S_{58}$ &
\begin{tabular}{ccc}
1 & 1 & 3 \\
1 & 2 & 3 \\
3 & 3 & 3
\end{tabular} &
\begin{tabular}{ccc}
3 & 3 & 3 \\
2 & 2 & 2 \\
3 & 3 & 3
\end{tabular} &
\begin{tabular}{c}
$xy\approx xz$, \\
$x+x^2\approx x^2$
\end{tabular} \\
\hline
$S_{60}$ &
\begin{tabular}{ccc}
1 & 1 & 3 \\
1 & 2 & 3 \\
3 & 3 & 3
\end{tabular} &
\begin{tabular}{ccc}
3 & 3 & 3 \\
3 & 2 & 3 \\
3 & 3 & 3
\end{tabular} &
\begin{tabular}{c}
$x^3\approx x^2$, $x^2+y^2\approx xy$, \\
$x+x^2\approx x^2$
\end{tabular} \\
\hline
\end{tabular}
\end{table}

There are, up to isomorphism, precisely $61$ ai-semirings of order three.
Following~\cite{zrc} we denote them by $S_i$ ($1 \leq i \leq 61$).
Five of these, specifically $S_{54}$, $S_{56}$, $S_{57}$, $S_{58}$ and $S_{60}$,
will be needed in subsequent sections.
The underlying set of each semiring is $\{1, 2, 3\}$.
Their addition and multiplication tables, together with finite equational bases,
are listed in Table~\ref{tb:3-element-ai-semirings}.
We conclude this section by presenting solution to the equational problem
for $S_{54}$, $S_{57}$, and $S_{60}$,
which can be obtained using \cite[Lemma 3.1]{yrzs} and its dual, together with \cite[Lemma 6.1]{yrzs}.

\begin{lem}\label{lemm54}
Let $\bu\approx \bu+\bq$ be a nontrivial ai-semiring identity, where $\bu=\bu_1+\bu_2+\cdots+\bu_n$ with $\bu_i, \bq \in X^+$ for $1\leq i \leq n$.
If $S_{54}$ satisfies $\bu\approx \bu+\bq$, then there exists $\bu_s \in \bu$ such that $\ell(\bu_s)\geq 2$,
$c(s(\bq))\subseteq c(s(\bu))$, and $h(\bq)\in c(\bu)$.
\end{lem}

\begin{lem}\label{lemm57}
Let $\bu\approx \bu+\bq$ be a nontrivial ai-semiring identity, where $\bu=\bu_1+\bu_2+\cdots+\bu_n$ with $\bu_i, \bq \in X^+$ for $1\leq i \leq n$.
If $S_{57}$ satisfies $\bu\approx \bu+\bq$, then there exists $\bu_t \in \bu$ such that $\ell(\bu_t)\geq 2$,
$c(p(\bq))\subseteq c(p(\bu))$, and $t(\bq)\in c(\bu)$.
\end{lem}

\begin{lem}\label{lemm60}
Let $\bu\approx \bu+\bq$ be a nontrivial ai-semiring identity, where $\bu=\bu_1+\bu_2+\cdots+\bu_n$ with $\bu_i, \bq \in X^+$ for $1\leq i \leq n$.
If $S_{60}$ satisfies $\bu\approx \bu+\bq$, then $\bu$ and $\bq$ simultaneously satisfy the following conditions$\colon$
\begin{itemize}
\item[$(1)$] There exists $\bu_j \in \bu$ such that $\ell(\bu_j)\geq 2$;

\item[$(2)$] If $\ell(\bq)=1$, then $c(\bq)\subseteq c(\bu)$;

\item[$(3)$] If $\ell(\bq)\geq 2$, then $c(\bq)\subseteq c(L_{\geq2}(\bu))$.
\end{itemize}
\end{lem}

\section{Equational basis of matrix semiring over a $2$-element ai-semiring}
In this section, we show that for each $S \in \{L_2, R_2, N_2, T_2\}$,
the varieties $\mathsf{V}(\mathbf{M}_n(S))$ for all $n \geq 2$ coincide and are finitely based.
These results are easy in comparison to those in Section 4.

\begin{pro}\label{pro21}
Let $n\geq 2$ be an integer. Then
$\mathsf{V}(\mathbf{M}_n(L_2))$ is the ai-semiring variety determined by the identities \eqref{eq:T22} and
\begin{align}
&xy\approx xz.    \label{f1}
\end{align}
\end{pro}
\begin{proof}
We first show that $\mathbf{M}_n(L_2)$ satisfies identities \eqref{eq:T22} and \eqref{f1}.
Let $A = [a_{ij}]$, $B = [b_{ij}]$, $C = [c_{ij}]$ be arbitrary matrices in $\mathbf{M}_n(L_2)$, where all entries belong to $L_2$.
For any indices $1 \leq i, j \leq n$, we have
\[
(AB)_{ij} = \sum_{k=1}^{n} a_{ik}b_{kj} \stackrel{\eqrefbasis{eq:L2}}=
\sum_{k=1}^{n} a_{ik} \stackrel{\eqrefbasis{eq:L2}}= \sum_{k=1}^{n} a_{ik}c_{kj} = (AC)_{ij},
\]
and
\begin{align*}
(A^2)_{ij} = \sum_{k=1}^{n} a_{ik}a_{kj}
&= \left( \sum_{k=1}^{n} a_{ik}a_{kj} \right) + a_{ij}a_{jj} \\
&\stackrel{\eqrefbasis{eq:L2}}= \left( \sum_{k=1}^{n} a_{ik}a_{kj} \right) + a_{ij} \\
&= (A^2)_{ij} + A_{ij} = (A^2 + A)_{ij}.
\end{align*}
This shows that $AB = AC$ and $A^2 = A^2 + A$.
Hence $\mathbf{M}_n(L_2)$ satisfies the identities \eqref{eq:T22} and \eqref{f1}.

Next, we show that every ai-semiring identity of $\mathbf{M}_n(L_2)$
is derivable from \eqref{eq:T22}, \eqref{f1} and the defining identities of $\mathbf{AI}$.
Let $\mathbf{u} \approx \mathbf{u} + \mathbf{q}$ be such a nontrivial identity,
where $\mathbf{u} = \mathbf{u}_1 + \mathbf{u}_2 + \dots + \mathbf{u}_m$ with $\mathbf{u}_i, \mathbf{q} \in X^+$ for $1 \leq i \leq m$.
Since $L_2$ embeds into $\mathbf{M}_n(L_2)$,
we have that $h(\mathbf{q}) = h(\mathbf{u}_i)$ for some $\mathbf{u}_i\in \mathbf{u}$.
We consider two cases.

\noindent \textbf{Case 1:} $\ell(\mathbf{q}) = 1$.
Then $\mathbf{u}_i = h(\mathbf{q}) s(\mathbf{u}_i) = \mathbf{q} s(\mathbf{u}_i)$.
Since the identity is nontrivial,
it follows that $s(\mathbf{u}_i)$ is nonempty. Consequently,
\[
\mathbf{u} \approx \mathbf{u} + \mathbf{u}_i \approx \mathbf{u} + \mathbf{q} s(\mathbf{u}_i)
\stackrel{\eqref{f1}}{\approx} \mathbf{u} + \mathbf{q}^2
\stackrel{\eqref{eq:T22}}{\approx} \mathbf{u} + \mathbf{q}^2 + \mathbf{q}.
\]

\noindent \textbf{Case 2:} $\ell(\mathbf{q}) \geq 2$.
Suppose, for contradiction, that $h(\mathbf{q}) \cap h(L_{\geq 2}(\mathbf{u})) = \emptyset$.
If $h(\mathbf{u}_i) = h(\mathbf{q})$ for some $\mathbf{u}_i\in \mathbf{u}$,
then $\mathbf{u}_i$ must be a single letter.
Define a substitution $\varphi \colon X \to \mathbf{M}_n(L_2)$ by
\[
\varphi(x) =
\begin{cases}
I_n, & \text{if } x = h(\mathbf{q}), \\
O_n, & \text{otherwise},
\end{cases}
\]
where $I_n$ is the $n\times n$ diagonal matrix with $1$ on the main diagonal and $0$ elsewhere,
and $O_n$ denotes the $n\times n$ constant matrix with all entries equal to $0$.
Then $\varphi(\mathbf{u}) = I_n$ while $\varphi(\mathbf{q}) = J_n$,
where $J_n$ denotes the $n\times n$ constant matrix with all entries equal to $1$.
Note that $I_n \neq J_n$, since $n \geq 2$.
This implies that $\varphi(\mathbf{u}) \neq \varphi(\mathbf{u} + \mathbf{q})$, a contradiction.
Hence there exists a word $\mathbf{u}_j \in L_{\geq 2}(\mathbf{u})$ such that $\mathbf{u}_j = h(\mathbf{q}) s(\mathbf{u}_j)$, and therefore
\[
\mathbf{u} \approx \mathbf{u} + \mathbf{u}_j \approx \mathbf{u} + h(\mathbf{q}) s(\mathbf{u}_j)
\stackrel{\eqref{f1}}{\approx} \mathbf{u} + h(\mathbf{q}) s(\mathbf{q}) \approx \mathbf{u} + \mathbf{q}.
\]

In both cases the identity $\mathbf{u} \approx \mathbf{u} + \mathbf{q}$ is derivable, which completes the proof.
\end{proof}

\begin{remark}\label{rem1}
The variety $\mathsf{V}(L_2)$ is a proper subvariety of $\mathsf{V}(\mathbf{M}_2(L_2))$,
since the identity $xy\approx x$ is satisfied by $L_2$, but fails in $\mathbf{M}_2(L_2)$.
Moreover, using the equational basis for $S_{58}$ listed in Table~\ref{tb:3-element-ai-semirings},
one finds that
\[
\mathsf{V}(\mathbf{M}_n(L_2)) = \mathsf{V}(S_{58})
\]
for all $n \geq 2$.
\end{remark}

\begin{pro}\label{pro22}
Let $n\geq 2$ be an integer. Then
$\mathsf{V}(\mathbf{M}_n(R_2))$ coincides with $\mathsf{V}(S_{56})$, properly contains $\mathsf{V}(R_2)$,
and is the ai-semiring variety determined by the identities \eqref{eq:T22} and
\begin{align}
&xy\approx zy.    \label{f3}
\end{align}
\end{pro}
\begin{proof}
Observe that $R_2$ and $L_2$ have dual multiplications.
The remaining argument parallels those of Proposition~\ref{pro21} and Remark~\ref{rem1}.
\end{proof}

\begin{pro}\label{pro233}
Let $n\geq 1$ be an integer. Then
$\mathsf{V}(\mathbf{M}_n(N_2))$ is the ai-semiring variety determined by the identities \eqref{eq:N21} and \eqref{eq:N22}.
\end{pro}
\begin{proof}
Let $A = [a_{ij}]$, $B = [b_{ij}]$, $C = [c_{ij}]$, $D = [d_{ij}]$ be arbitrary matrices in $\mathbf{M}_n(N_2)$.
For any $1 \leq i, j \leq n$, we have
\[
(AB)_{ij} = \sum_{k=1}^{n} a_{ik}b_{kj} \stackrel{\eqrefbasis{eq:N21}}= \sum_{k=1}^{n} c_{ik}d_{kj} = (CD)_{ij},
\]
and
\begin{align*}
(A + A^2)_{ij}
&= A_{ij} + (A^2)_{ij}=a_{ij} + \left(\sum_{k=1}^n a_{ik}a_{kj}\right)\\
&\stackrel{\eqrefbasis{eq:N21}}= a_{ij} + a_{ij}a_{ij}
\stackrel{\eqrefbasis{eq:N22}}= a_{ij}= A_{ij}.
\end{align*}
This shows that $AB = CD$ and $A=A+A^2$.
Hence $\mathbf{M}_n(N_2)$ satisfies both identities \eqref{eq:N21} and \eqref{eq:N22}.

According to Table~\ref{tab:2-element-ai-semirings}, the ai-semiring variety $\mathsf{V}(N_2)$ is
defined by the identities \eqref{eq:N21} and \eqref{eq:N22}.
Since $N_2$ can be embedded into $\mathbf{M}_n(N_2)$,
we conclude that $\mathsf{V}(\mathbf{M}_n(N_2))$ is the same variety.
\end{proof}

\begin{pro}\label{pro23}
Let $n\geq 1$ be an integer. Then
$\mathsf{V}(\mathbf{M}_n(T_2))$ is the ai-semiring variety determined by the identities \eqref{eq:N21} and \eqref{eq:T22}.
\end{pro}

\begin{proof}
First, $\mathbf{M}_n(T_2)$ satisfies identity~\eqref{eq:N21}; the argument is similar to that of Proposition~\ref{pro233}.
Now let $A = [a_{ij}]$ be an arbitrary matrix in $\mathbf{M}_n(T_2)$.
For all indices $1 \leq i, j \leq n$, we have
\begin{align*}
(A^2)_{ij} = \sum_{k=1}^n a_{ik}a_{kj}
&= \left(\sum_{k=1}^n a_{ik}a_{kj}\right) + a_{ij}a_{jj} \\
&\stackrel{\eqrefbasis{eq:N21}}= \left(\sum_{k=1}^n a_{ik}a_{kj}\right) + a_{ij}^2 \\
&\stackrel{\eqrefbasis{eq:T22}}= \left(\sum_{k=1}^n a_{ik}a_{kj}\right) + a_{ij}^2 + a_{ij} \\
&\stackrel{\eqrefbasis{eq:N21}}= (A^2)_{ij} + A_{ij} = (A^2 + A)_{ij}.
\end{align*}
Thus $A^2= A^2+A$, and so $\mathbf{M}_n(T_2)$ also satisfies the identity~\eqref{eq:T22}.

From Table~\ref{tab:2-element-ai-semirings}, the ai-semiring variety $\mathsf{V}(T_2)$ is
defined by the identities \eqref{eq:N21} and \eqref{eq:T22}.
Since $T_2$ can be embedded into $\mathbf{M}_n(T_2)$,
we therefore obtain that $\mathsf{V}(\mathbf{M}_n(T_2))$ is the same variety.
\end{proof}

\section{Equational basis of the matrix semiring $\mathbf{M}_n(M_2)$}
In this section we study the finite basis problem for the matrix semiring $\mathbf{M}_n(M_2)$ with $n \geq 2$.
We show that all these semirings generate the same finitely based variety,
which is in fact generated by a single six-element ai-semiring.

We begin with the case $\mathbf{M}_2(M_2)$. Its $16$ elements can be classified as follows:
\begin{itemize}
\item $O$: the constant matrix with all entries equal to $0$;

\item $A, B, C, D$: the matrices with exactly one entry equal to $1$,
namely
\[
A=
\begin{bmatrix}
0&0\\
1&0
\end{bmatrix},\
B=
\begin{bmatrix}
0&1\\
0&0
\end{bmatrix},\
C=
\begin{bmatrix}
1&0\\
0&0
\end{bmatrix},\
D=
\begin{bmatrix}
0&0\\
0&1
\end{bmatrix};
\]

\item $E, G, P, Q, R, S$: the matrices with exactly two entries equal to $1$;

\item $W, X, Y, Z$: the matrices with exactly three entries equal to $1$;

\item $F$: the constant matrix with all entries equal to $1$.
\end{itemize}
The additive order of $\mathbf{M}_2(M_2)$ is shown in Figure~\ref{fig:mn_m2_semilattice}, and its multiplication is given by Table~\ref{tab:multR2}.
A direct verification shows that the set $\{O, A, P, R, Z, F\}$ forms a subsemiring of $\mathbf{M}_2(M_2)$,
which is denoted by $\mathbf{SR}_6$.
For the reader's convenience,
the additive order and multiplicative table of $\mathbf{SR}_6$
are provided in Figure~\ref{fig:additive_order 6} and Table~\ref{tab:cayley_table 6}, respectively.

The following result shows that $\mathbf{SR}_6$ is finitely based.
\begin{figure}[htbp]
    \centering
\begin{tikzpicture}[scale=1.2]
    \tikzset{smallnode/.style={font=\tiny}} 
    \node at (2.4,4) [circle,fill=red,draw=red, inner sep=1.5pt,label={[smallnode]above:$F$}] (F) {};
    \node at (0,3) [circle,fill=red,draw=red, inner sep=1.5pt,label={[smallnode]above:$Z$}] (Z) {};
    \node at (1.6,3) [circle,fill,inner sep=1.5pt,label={[smallnode]above:$Y$}] (Y) {};
    \node at (3.2,3) [circle,fill,inner sep=1.5pt,label={[smallnode]above:$X$}] (X) {};
    \node at (4.8,3) [circle,fill,inner sep=1.5pt,label={[smallnode]above:$W$}] (W) {};
    \node at (-1.6,2) [circle,fill=red,draw=red, inner sep=1.5pt,label={[smallnode]left:$P$}] (P) {};
    \node at (0,2) [circle,fill=red,draw=red, inner sep=1.5pt,label={[smallnode]below:$R$}] (R) {};
    \node at (1.6,2) [circle,fill,inner sep=1.5pt,label={[smallnode]below:$G$}] (G) {};
    \node at (3.2,2) [circle,fill,inner sep=1.5pt,label={[smallnode]below:$E$}] (E) {};
    \node at (4.8,2) [circle,fill,inner sep=1.5pt,label={[smallnode]below:$S$}] (S) {};
    \node at (6.4,2) [circle,fill,inner sep=1.5pt,label={[smallnode]right:$Q$}] (Q) {};
    \node at (0,1) [circle,fill=red,draw=red, inner sep=1.5pt,label={[smallnode]below:$A$}] (A) {};
    \node at (1.6,1) [circle,fill,inner sep=1.5pt,label={[smallnode]below:$B$}] (B) {};
    \node at (3.2,1) [circle,fill,inner sep=1.5pt,label={[smallnode]below:$C$}] (C) {};
    \node at (4.8,1) [circle,fill,inner sep=1.5pt,label={[smallnode]below:$D$}] (D) {};
    \node at (2.4,0) [circle,fill=red,draw=red, inner sep=1.5pt,label={[smallnode]below:$O$}] (0node) {};

    \draw (F) -- (Z) (F) -- (Y) (F) -- (X) (F) -- (W);
    \draw (Z) -- (P) (Z) -- (R) (Z) -- (E);
    \draw (Y) -- (E) (Y) -- (Q) (Y) -- (S);
    \draw (X) -- (P) (X) -- (G) (X) -- (S);
    \draw (W) -- (R) (W) -- (G) (W) -- (Q);
    \draw (A) -- (P) (A) -- (R) (A) -- (G) (A) -- (0node);
    \draw (B) -- (S) (B) -- (Q) (B) -- (G) (B) -- (0node);
    \draw (C) -- (P) (C) -- (S) (C) -- (E) (C) -- (0node);
    \draw (D) -- (Q) (D) -- (R) (D) -- (E) (D) -- (0node);
\end{tikzpicture}
\caption{The additive order of $\mathbf{M}_2(M_2)$}
\label{fig:mn_m2_semilattice}
\end{figure}
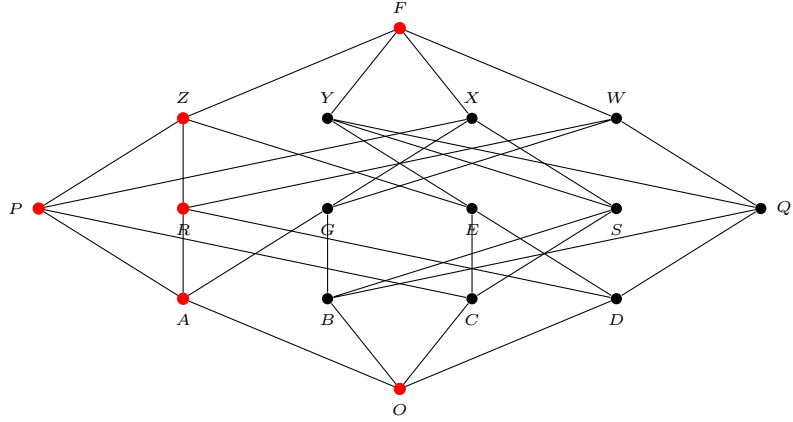

\begin{table}[ht]
\centering
\scriptsize 
\setlength{\tabcolsep}{2.8pt} 
\renewcommand{\arraystretch}{1.337} 
\caption{The multiplicative table of $\mathbf{M}_2(M_2)$}
\label{tab:multR2}
\begin{tabular}{c|*{16}{c}}
$\cdot$ & $\textcolor{red}{O}$ & $\textcolor{red}{A}$ & $B$ & $C$ & $D$ & $E$ & $G$ & $\textcolor{red}{P}$ & $Q$ & $\textcolor{red}{R}$ & $S$ & $\textcolor{red}{Z}$ & $Y$ & $X$ & $W$ & $\textcolor{red}{F}$ \\
\hline
$\textcolor{red}{O}$     & $\textcolor{red}{O}$ & $\textcolor{red}{P}$ & $Q$  & $P$ & $Q$ & $F$  & $F$ & $\textcolor{red}{P}$ & $Q$  & $\textcolor{red}{F}$ & $F$ & $\textcolor{red}{F}$ & $F$ & $F$ & $F$ & $\textcolor{red}{F}$  \\
$\textcolor{red}{A}$    & $\textcolor{red}{R}$ & $\textcolor{red}{Z}$   & $W$ & $Z$ & $W$ & $F$ & $F$ & $\textcolor{red}{Z}$ & $W$ & $\textcolor{red}{F}$ & $F$ & $\textcolor{red}{F}$ & $F$ & $F$ & $F$ & $\textcolor{red}{F}$ \\
$B$                     & $S$ & $X$ & $Y$ & $X$ & $Y$ & $F$ & $F$ & $X$ & $Y$ & $F$ & $F$ & $F$ & $F$ & $F$ & $F$ & $F$ \\
$C$                     & $S$ & $X$ & $Y$ & $X$ & $Y$ & $F$ & $F$ & $X$ & $Y$ & $F$ & $F$ & $F$ & $F$ & $F$ & $F$ & $F$ \\
$D$                     & $R$ & $Z$   & $W$ & $Z$ & $W$ & $F$ & $F$ & $Z$ & $W$ & $F$ & $F$ & $F$ & $F$ & $F$ & $F$ & $F$ \\
$E$                     & $F$ & $F$ & $F$ & $F$ & $F$ & $F$ & $F$ & $F$ & $F$ & $F$ & $F$ & $F$ & $F$ & $F$ & $F$ & $F$\\
$G$                     & $F$ & $F$ & $F$ & $F$ & $F$ & $F$ & $F$ & $F$ & $F$ & $F$ & $F$ & $F$ & $F$ & $F$ & $F$ & $F$ \\
$\textcolor{red}{P}$    & $\textcolor{red}{F}$ & $\textcolor{red}{F}$ & $F$ & $F$ & $F$ & $F$ & $F$ & $\textcolor{red}{F}$ & $F$ & $\textcolor{red}{F}$ & $F$ & $\textcolor{red}{F}$ & $F$ & $F$ & $F$ & $\textcolor{red}{F}$ \\
$Q$                     & $F$ & $F$ & $F$ & $F$ & $F$ & $F$ & $F$ & $F$ & $F$ & $F$ & $F$ & $F$ & $F$ & $F$ & $F$ & $F$ \\
$\textcolor{red}{R}$    & $\textcolor{red}{R}$ & $\textcolor{red}{Z}$   & $W$ & $Z$ & $W$ & $F$ & $F$ & $\textcolor{red}{Z}$ & $W$ & $\textcolor{red}{F}$ & $F$ & $\textcolor{red}{F}$ & $F$ & $F$ & $F$ & $\textcolor{red}{F}$ \\
$S$                    & $S$ & $X$ & $Y$ & $X$ & $Y$ & $F$ & $F$ & $X$ & $Y$ & $F$ & $F$ & $F$ & $F$ & $F$ & $F$ & $F$ \\
$\textcolor{red}{Z}$    & $\textcolor{red}{F}$ & $\textcolor{red}{F}$ & $F$ & $F$ & $F$ & $F$ & $F$ & $\textcolor{red}{F}$ & $F$ & $\textcolor{red}{F}$ & $F$ & $\textcolor{red}{F}$ & $F$ & $F$ & $F$ & $\textcolor{red}{F}$  \\
$Y$                    & $F$ & $F$ & $F$ & $F$ & $F$ & $F$ & $F$ & $F$ & $F$ & $F$ & $F$ & $F$ & $F$ & $F$ & $F$ & $F$ \\
$X$                    & $F$ & $F$ & $F$ & $F$ & $F$ & $F$ & $F$ & $F$ & $F$ & $F$ & $F$ & $F$ & $F$ & $F$ & $F$ & $F$ \\
$W$                    & $F$ & $F$ & $F$ & $F$ & $F$ & $F$ & $F$ & $F$ & $F$ & $F$ & $F$ & $F$ & $F$ & $F$ & $F$ & $F$ \\
$\textcolor{red}{F}$   & $\textcolor{red}{F}$ & $\textcolor{red}{F}$ & $F$ & $F$ & $F$ & $F$ & $F$ & $\textcolor{red}{F}$ & $F$ & $\textcolor{red}{F}$ & $F$ & $\textcolor{red}{F}$ & $F$ & $F$ & $F$ & $\textcolor{red}{F}$ \\
\end{tabular}
\end{table}

\begin{figure}
\centering
\begin{minipage}{0.5\textwidth}
\centering
\setlength{\unitlength}{0.7cm}
\begin{picture}(20, 4)
\put(4.5,2.2){\line(1,-1){1}}
\put(4.5,3.2){\line(0,-1){1}}
\put(4.5,2.2){\line(-1,-1){1}}
\put(4.5,0.2){\line(1,1){1}}
\put(4.5,0.2){\line(-1,1){1}}
\put(4.5,0.2){\line(0,-2){1}}

\multiput(4.5,2.2)(1,-1){2}{\circle*{0.1}}
\multiput(4.5,3.2)(1,-1){1}{\circle*{0.1}}
\multiput(4.5,2.2)(-1,-1){2}{\circle*{0.1}}
\multiput(4.5,0.2)(1,1){2}{\circle*{0.1}}
\put(4.5,-0.8){\circle*{0.1}}

\put(4.7,2.4){\makebox(0.2,0){$Z$}}
\put(4.7,3.2){\makebox(0.2,0){$F$}}
\put(4.7,0){\makebox(0.1,0.1){$A$}}
\put(3.3,1.2){\makebox(-0.2,0){$P$}}
\put(5.7,1.2){\makebox(0.2,0){$R$}}
\put(4.7,-1){\makebox(0.2,0){$O$}}
\end{picture}
\vspace{0.6em}
\caption{The additive order of $\mathbf{SR}_6$}
\label{fig:additive_order 6}
\end{minipage}
\hfill
\begin{minipage}{0.49\textwidth}
\centering
\captionof{table}{The multiplicative table of $\mathbf{SR}_6$}
\label{tab:cayley_table 6}
\vspace{0.5em}
\begin{tabular}{c|cccccc}
$\cdot$ & $O$ & $A$ & $P$ & $R$ & $Z$ & $F$\\
\hline
$O$ & $O$ & $P$ & $P$ & $F$ & $F$ & $F$\\
$A$ & $R$ & $Z$ & $Z$ & $F$ & $F$ & $F$\\
$P$ & $F$ & $F$ & $F$ & $F$ & $F$ & $F$\\
$R$ & $R$ & $Z$ & $Z$ & $F$ & $F$ & $F$\\
$Z$ & $F$ & $F$ & $F$ & $F$ & $F$ & $F$\\
$F$ & $F$ & $F$ & $F$ & $F$ & $F$ & $F$\\
\end{tabular}
\end{minipage}
\end{figure}

\begin{pro}\label{lemm24}
$\mathsf{V}(\mathbf{SR}_6)$ is the ai-semiring variety defined by the identities
\begin{align}
&xy\approx xy+x;    \label{f9}\\
&xy \approx xy+y;                 \label{f10}\\
&x_1x_2+x_3x_4 \approx x_1x_2+x_3x_4+x_1x_4;                  \label{f11}\\
&x_1x_2x_3x_4 \approx x_1x_2x_3+x_1x_2x_4+x_1x_3x_4+x_2x_3x_4.    \label{f12}
\end{align}
\end{pro}

\begin{proof}
It is straightforward to verify that $\mathbf{SR}_6$ satisfies the identities \eqref{f9}--\eqref{f12}.
In the remainder it suffices to show that every nontrivial ai-semiring identity satisfied by $\mathbf{SR}_6$ is derivable from \eqref{f9}--\eqref{f12} and the identities defining $\mathbf{AI}$.
Let $\bu \approx \bu+\bq$ be such an identity, where $\bu=\bu_1+\bu_2+\cdots+\bu_m$ and $\bu_i, \bq\in X^+$, $1 \leq i \leq m$.
Since $S_{54} \cong \{R, O, F\}$ and $S_{57} \cong \{P, O, F\}$,
Lemmas~\ref{lemm54} and~\ref{lemm57} imply
\[
c(p(\mathbf{q})) \subseteq c(p(\mathbf{u})) \quad \text{and} \quad c(s(\mathbf{q})) \subseteq c(s(\mathbf{u})).
\]
Moreover, because $S_{60} \cong \{O, Z, F\}$, Lemma~\ref{lemm60} gives
\[
c(\mathbf{q}) \subseteq
\begin{cases}
c(\mathbf{u}), & \text{if } \ell(\mathbf{q}) = 1,\\[4pt]
c(L_{\geq 2}(\mathbf{u})), & \text{if } \ell(\mathbf{q}) \geq 2.
\end{cases}
\]
If $\bq=x_1x_2\cdots x_k$ with $k\geq 4$, then by the identity \eqref{f12},
one can derive
\[\bq\approx \sum_{1\leq {i_1}<{i_2}<{i_3}\leq k}x_{i_1}x_{i_2}x_{i_3}.
\]
Hence, we only need to consider the case that $\ell(\bq) \leq 3$.

\noindent \textbf{Case 1.} $\ell(\bq)=1$. Then there exists $\bu_r\in \bu$ with $\ell(\bu_r)\geq 2$ such that
$c(\bq)\subseteq c(\bu_r)$,
and so $\bu_r=\bp_1\bq\bp_2$ for some $\bp_1, \bp_2\in X^*$.
We therefore have
\[
\bu \approx \bu+\bu_r \approx \bu+\bp_1\bq\bp_2\stackrel{\eqref{f9}\eqref{f10}}\approx \bu+\bq.
\]
This derives the identity $\bu\approx \bu+\bq$.

\noindent \textbf{Case 2.} $\ell(\bq)=2$. Then $c(\bq)\subseteq c(L_{\geq 2}(\bu))$.
Let us write $\bq=xy$ with $x, y\in X$. Then $x\in c(p(\bu))$ and $y\in c(s(\bu))$.
Hence,
there exist $\bu_i, \bu_j\in L_{\geq 2}(\bu)$ such that $\bu_i=\bp_1x\bp_1'$ and $\bu_j=\bp_2y\bp_2'$
for some $\bp_1', \bp_2\in X^+$ and some $\bp_1, \bp_2' \in X^*$.
We can deduce that
\begin{align*}
\bu
&\approx \bu+\bu_i+\bu_j\\
&\approx \bu+\bp_1x\bp_1'+\bp_2y\bp_2'\\
&\approx \bu+\bp_1x\bp_1'+\bp_2y\bp_2'+\bp_1xy\bp_2' &&(\text{by}~\eqref{f11})\\
&\approx \bu+\bp_1x\bp_1'+\bp_2y\bp_2'+\bp_1xy\bp_2'+xy &&(\text{by}~\eqref{f9}, \eqref{f10})\\
&\approx \bu+\bp_1x\bp_1'+\bp_2y\bp_2'+\bp_1xy\bp_2'+\bq.
\end{align*}
This implies the identity $\bu\approx \bu+\bq$.

\noindent \textbf{Case 3.} $\ell(\bq)=3$.
We may write $\bq=x_1x_2x_3$ for $x_1, x_2, x_3\in X$.
Then $x_1\in c(p(\bu))$, $x_3\in c(s(\bu))$ and $x_2\in c(p(\bu))\cap c(s(\bu))$.
We claim that there exists $\bu_i\in \bu$ such that $\bu_i=\bp_1x_2\bp_2$,
where both words $\bp_1$ and $\bp_2$ are nonempty.

Suppose, to the contrary, that for every $\mathbf{u}_i \in \mathbf{u}$ of
the form $\mathbf{u}_i = \mathbf{p}_1 x_2 \mathbf{p}_2$, at least one of $\mathbf{p}_1$ or $\mathbf{p}_2$ is empty.
Consider the substitution $\varphi\colon X \to \mathbf{SR}_6$ defined by
$\varphi(x)=A$ if $x=x_2$, and $\varphi(x)=O$ otherwise.
Then $\varphi(\bu)\in\{P, R, Z\}$ while $\varphi(\bq)=F$.
This implies that $\varphi(\bu) \neq \varphi(\bu+\bq)$, a contradiction.
Since $x_1\in c(p(\bu))$ and $x_3\in c(s(\bu))$, there are $\bu_s, \bu_t\in L_{\geq 2}(\bu)$ such that
$\bu_s=\bp_3x_1\bp_3'$ and $\bu_t=\bp_4x_3\bp_4'$ for some $\bp_3', \bp_4\in X^+$ and some $\bp_3, \bp_4' \in X^*$.
We deduce that
\begin{align*}
\bu
&\approx \bu+\bu_s+\bu_i+\bu_t\\
&\approx \bu+\bp_3x_1\bp_3'+\bp_1x_2\bp_2+\bp_4x_3\bp_4'\\
&\approx \bu+\bp_3x_1\bp_3'+\bp_1x_2\bp_2+\bp_3x_1x_2\bp_2+\bp_4x_3\bp_4' &&(\text{by}~\eqref{f11})\\
&\approx \bu+\bp_3x_1\bp_3'+\bp_1x_2\bp_2+\bp_3x_1x_2\bp_2+\bp_4x_3\bp_4'+\bp_3x_1x_2x_3\bp_4' &&(\text{by}~\eqref{f11})\\
&\approx \bu+\bp_3x_1x_2x_3\bp_4'+x_1x_2x_3&&(\text{by}~\eqref{f9}, \eqref{f10})\\
&\approx \bu+\bp_3x_1x_2x_3\bp_4'+\bq.
\end{align*}
This completes the proof.
\end{proof}

\begin{cor}\label{coro25122001}
$\mathsf{V}(\mathbf{M}_2(M_2))=\mathsf{V}(\mathbf{SR}_6)$.
\end{cor}
\begin{proof}
A routine verification shows that $\mathbf{M}_2(M_2)$ satisfies identities \eqref{f9}--\eqref{f12}.
By Proposition~\ref{lemm24},
$\mathsf{V}(\mathbf{M}_2(M_2)) \subseteq \mathsf{V}(\mathbf{SR}_6)$.
Since $\mathbf{SR}_6$ is a subsemiring of $\mathbf{M}_2(M_2)$, it follows that
$\mathsf{V}(\mathbf{SR}_6) \subseteq \mathsf{V}(\mathbf{M}_2(M_2))$.
Therefore, $\mathsf{V}(\mathbf{M}_2(M_2)) = \mathsf{V}(\mathbf{SR}_6)$.
\end{proof}
%
\begin{cor}\label{cor25}
$\mathsf{V}(\mathbf{M}_2(M_2))$ is the ai-semiring variety defined by the identities \eqref{f9}--\eqref{f11}.
\end{cor}
\begin{proof}
By Proposition~\ref{lemm24} and Corollary~\ref{coro25122001},
it suffices to show that the identity \eqref{f12} follows from the identities \eqref{f9}--\eqref{f11}.
Indeed, we have
\begin{align*}
x_1x_2x_3x_4 &\approx x_1x_2x_3x_4+x_2x_3x_4&&(\text{by}~\eqref{f10})\\
             &\approx x_1x_2x_3x_4+x_1x_3x_4+x_2x_3x_4&&(\text{by}~\eqref{f11})\\
             &\approx x_1x_2x_3x_4+x_1x_3x_4+x_2x_3x_4+x_1x_2x_3&&(\text{by}~\eqref{f9}\\
             &\approx x_1x_2x_3x_4+x_1x_2x_3+x_2x_3x_4+x_1x_3x_4+x_1x_2x_4&&(\text{by}~\eqref{f11})\\
             &\approx x_1x_2x_3+x_2x_3x_4+x_1x_3x_4+x_1x_2x_4.&&(\text{by}~\eqref{f11})
\end{align*}
This derives the identity \eqref{f12}.
\end{proof}

\begin{pro}\label{pro25}
Let $n\geq 2$ be an integer. Then $\mathsf{V}(\mathbf{M}_n(M_2))=\mathsf{V}(\mathbf{M}_2(M_2))$.
\end{pro}
\begin{proof}
We first show that $\mathbf{M}_n(M_2)$ lies in the variety $\mathsf{V}(\mathbf{M}_2(M_2))$.
By Corollary~\ref{cor25} it is enough to verify that $\mathbf{M}_n(M_2)$ satisfies the identities \eqref{f9}--\eqref{f11}.
Let $A=[a_{ij}]$, $B=[a_{ij}]$, $C=[a_{ij}]$, and $D =[d_{ij}]$ be arbitrary matrices in $\mathbf{M}_n(M_2)$.
For any $1 \leq i, j \leq n$,
\begin{align*}
(AB+A)_{ij}
&=(AB)_{ij}+A_{ij}\\
&=\left(\sum_{k=1}^{n} a_{ik}b_{kj}\right)+a_{ij}
\stackrel{\eqrefbasis{eq:M2}}=\left(\sum_{k=1}^n a_{ik}\right) + \left(\sum_{k=1}^n b_{kj}\right) + a_{ij}\\
&=\left(\sum_{k=1}^n a_{ik}\right) + \left(\sum_{k=1}^n b_{kj}\right)\stackrel{\eqrefbasis{eq:M2}}=\sum_{k=1}^{n} a_{ik}b_{kj}=(AB)_{ij}.
\end{align*}
Thus $AB + A=AB$. A similar argument yields $AB + B=AB$.
So
$\mathbf{M}_n(M_2)$ satisfies the identities \eqref{f9} and \eqref{f10}.
Next,
\begin{align*}
& \quad(AB+CD)_{ij}\\
&=(AB)_{ij}+ (CD)_{ij}\\
&= \left(\sum_{k=1}^n a_{ik}b_{kj}\right) + \left(\sum_{k=1}^n c_{ik}d_{kj}\right) \\
&\stackrel{\eqrefbasis{eq:M2}}=\left(\sum_{k=1}^n a_{ik}\right) + \left(\sum_{k=1}^n b_{kj}\right) + \left(\sum_{k=1}^n c_{ik}\right) + \left(\sum_{k=1}^n d_{kj}\right) \\
&=\left(\sum_{k=1}^n a_{ik}\right) + \left(\sum_{k=1}^n b_{kj}\right) + \left(\sum_{k=1}^n c_{ik}\right) + \left(\sum_{k=1}^n d_{kj}\right) + \left(\sum_{k=1}^n a_{ik}\right) + \left(\sum_{k=1}^n d_{kj}\right) \\
&\stackrel{\eqrefbasis{eq:M2}}=\left(\sum_{k=1}^n a_{ik}b_{kj}\right) + \left(\sum_{k=1}^n c_{ik}d_{kj}\right) + \left(\sum_{k=1}^n a_{ik}d_{kj}\right) \\
&= (AB)_{ij}+ (CD)_{ij}+ (AD)_{ij}\\
&= (AB+CD+AD)_{ij}.
\end{align*}
Hence $AB + CD=AB + CD + AD$, and so $\mathbf{M}_n(M_2)$ satisfies the identity \eqref{f11}.

To end this proof, it remains to show that $\mathbf{M}_2(M_2)$ can be embedded into $\mathbf{M}_n(M_2)$.
We fix some notation: for any positive integers $k$ and $\ell$,
\begin{itemize}
\item $O_{k \times \ell}$ denotes the $k \times \ell$ constant matrix with all entries equal to $0$.

\item $J_{k \times \ell}$ denotes the $k \times \ell$ constant matrix with all entries equal to $1$.

\item $H_{2 \times (n-2)}$ denotes the $2 \times (n-2)$ constant matrix
$\begin{bmatrix}
0 & 0 & \cdots & 0 \\ 1 & 1 & \cdots & 1
\end{bmatrix}$.
\item $K_{2 \times (n-2)}$ denotes the $2 \times (n-2)$ constant matrix
$\begin{bmatrix}
1 & 1 & \cdots & 1 \\ 0 & 0 & \cdots & 0
\end{bmatrix}$.

\item $M_{k \times \ell}^T$ denotes the transpose of a matrix $M_{k \times \ell}$.
\end{itemize}
Consider the mapping $\varphi\colon \mathbf{M}_2(M_2) \to \mathbf{M}_n(M_2)$ defined by
\[
\begin{aligned}
\varphi(O) &= O_{n \times n}, \qquad \varphi(F)= J_{n \times n}, \\[1pt]
\varphi(X) &=
\begin{bmatrix}
X & J_{2 \times (n-2)} \\
J_{(n-2) \times 2} & J_{(n-2) \times (n-2)}
\end{bmatrix}\quad \text{for } X \in \{E, G, W, X, Y, Z\}, \\[8pt]
\varphi(A) &=\begin{bmatrix}
A & H_{2 \times (n-2)} \\
K_{2 \times (n-2)}^T & J_{(n-2) \times (n-2)}
\end{bmatrix},
\varphi(B) =\begin{bmatrix}
B &  K_{2 \times (n-2)}\\
H_{2 \times (n-2)}^T & J_{(n-2) \times (n-2)}
\end{bmatrix}, \\[8pt]
\varphi(C) &=\begin{bmatrix}
C &  K_{2 \times (n-2)}\\
K_{2 \times (n-2)}^T & J_{(n-2) \times (n-2)}
\end{bmatrix},
\varphi(D) =\begin{bmatrix}
D &  H_{2 \times (n-2)}\\
H_{2 \times (n-2)}^T & J_{(n-2) \times (n-2)}
\end{bmatrix}, \\[8pt]
\varphi(P) &=\begin{bmatrix}
P &  J_{2 \times (n-2)}\\
K_{2 \times (n-2)}^T & J_{(n-2) \times (n-2)}
\end{bmatrix},
\varphi(Q) =\begin{bmatrix}
Q &  J_{2 \times (n-2)}\\
H_{2 \times (n-2)}^T & J_{(n-2) \times (n-2)}
\end{bmatrix}, \\[8pt]
\varphi(R) &=\begin{bmatrix}
R &  H_{2 \times (n-2)}\\
J_{2 \times (n-2)}^T & J_{(n-2) \times (n-2)}
\end{bmatrix},
\varphi(S) =\begin{bmatrix}
S &  K_{2 \times (n-2)}\\
J_{2 \times (n-2)}^T & J_{(n-2) \times (n-2)}
\end{bmatrix}. \\[8pt]
\end{aligned}
\]
It is straightforward to verify that $\varphi$ is a monomorphism.
\end{proof}

\begin{cor}
Let $n\geq 2$ be an integer.
Then $\mathsf{V}(\mathbf{M}_n(M_2))$ is the ai-semiring variety defined by the identities \eqref{f9}--\eqref{f11}.
\end{cor}
\begin{proof}
This follows directly from Corollary~\ref{cor25} and Proposition~\ref{pro25}.
\end{proof}

\section{Conclusion}
We have provided a complete classification of
matrix semirings $\mathbf{M}_n(S)$ over two-element ai-semirings $S$
with respect to the finite basis property.
From Remark~\ref{rem1} and \cite[Figure 1]{yr25},
the variety $\mathsf{V}(\mathbf{M}_n(L_2))$ has precisely $5$ subvarieties;
they form a distributive lattice and are all finitely based.
The same conclusion holds for the variety $\mathsf{V}(\mathbf{M}_n(R_2))$.
By \cite[Theorem 1]{polin} together with Propositions~\ref{pro233} and \ref{pro23},
the varieties $\mathsf{V}(\mathbf{M}_n(N_2))$ and $\mathsf{V}(\mathbf{M}_n(T_2))$
are both minimal nontrivial varieties.
In contrast, the subvariety lattice of the variety $\mathsf{V}(\mathbf{M}_n(M_2))$
remains to be explicitly described.

Dolinka's result demonstrates that
the finite basis properties of $S$ and $\mathbf{M}_n(S)$ do not always coincide.
This leads to a natural question:
does there exist a nonfinitely based ai-semiring $S$ for which $\mathbf{M}_n(S)$ is finitely based?

\subsection*{Acknowledgment}
The authors would like to thank Zidong Gao, Simin Lyu, Chenyu Yang, Ting Yu, and Mengya Yue
for their helpful discussions and contributions to this work.
We are also grateful to Professor Mikhail V. Volkov for his valuable comments and suggestions, which significantly improved the paper.

\bibliographystyle{amsplain}


\end{document}